\documentclass{amsart}
\usepackage{helvet}  
\usepackage{hyperref}
\hypersetup{bookmarksnumbered} 
\usepackage{amssymb,amsmath,amsthm}
\usepackage{graphicx}
\usepackage[subrefformat=parens,labelformat=parens]{subfig}
\graphicspath{{Pictures/}}

\input xy
\xyoption{all}

\usepackage{color} 
\newcommand{\blue}[1]{\textcolor{blue}{#1}}     

\makeatletter

\let\c@table\c@figure
\makeatother

\newtheorem*{maintheorem}{Main Theorem}
\newtheorem{theorem}{Theorem}[section]
\newtheorem{proposition}[theorem]{Proposition}
\newtheorem{lemma}[theorem]{Lemma}
\newtheorem{corollary}[theorem]{Corollary}

\theoremstyle{definition}
\newtheorem{definition}[theorem]{Definition}

\theoremstyle{remark}
\newtheorem{remark}[theorem]{Remark}

\newcommand{\newword}[1]{\textbf{#1}}
\newcommand{\CLF}{\mathrm{CLF}}
\newcommand{\e}{\mathfrak{e}}
\newcommand{\f}{\mathfrak{f}}
\newcommand{\g}{\mathfrak{g}}
\newcommand{\h}{\mathfrak{h}}
\newcommand{\s}{\mathfrak{s}}
\renewcommand{\t}{\mathfrak{t}}
\renewcommand{\r}{\mathfrak{r}}
\renewcommand{\k}{\mathfrak{k}}

\newcommand{\cd}{\mathrm{cd}}

\begin{document}

\title[Conjugator Length in Thompson's Groups]{Conjugator Length in Thompson's Groups}
\author{James Belk}
\address{School of Mathematics \& Statistics, University of Glasgow, Glasgow G12 8QQ, United Kingdom.}
\email{\href{mailto:jim.belk@glasgow.ac.uk}{jim.belk@glasgow.ac.uk}}

\thanks{The first author has been partially supported by EPSRC grant EP/R032866/1 as well as the National Science Foundation under Grant No.~DMS-1854367 during the creation of this paper.}

\author{Francesco Matucci}
\address{Dipartimento di Matematica e Applicazioni, Universit\`{a} degli Studi di Milano--Bicocca, Milan 20125, Italy.}
\email{\href{mailto:francesco.matucci@unimib.it}{francesco.matucci@unimib.it}}
\thanks{The second author is a member of the Gruppo Nazionale per le Strutture Algebriche, Geometriche e le loro Applicazioni (GNSAGA) of the Istituto Nazionale di Alta Matematica (INdAM) and gratefully acknowledges the support of the 
Funda\c{c}\~ao para a Ci\^encia e a Tecnologia  (CEMAT-Ci\^encias FCT projects UIDB/04621/2020 and UIDP/04621/2020) and of the Universit\`a degli Studi di Milano--Bicocca
(FA project ATE-2017-0035 ``Strutture Algebriche'').
}

\begin{abstract}
We prove Thompson's group~$F$ has quadratic conjugator length function.  That is, for any two conjugate elements of~$F$ of length~$n$ or less, there exists an element of $F$ of length $O(n^2)$ that conjugates one to the other.  Moreover, there exist conjugate pairs of elements of $F$ of length at most $n$ such that the shortest conjugator between them has length~$\Omega(n^2)$. This latter statement holds for $T$ and $V$ as well.
\end{abstract}

\maketitle

Let $G$ be a group with finite generating set~$S$, and let $\ell\colon G\to\mathbb{N}$ be the word length function on $G$ with respect to~$S$.  If $g$ and $h$ are conjugate elements of $G$, the \newword{conjugator distance} from $g$ to~$h$ is
\[
\cd(g,h) = \min\{\ell(k) \mid k\in G\text{ and }h=k^{-1}gk\}.
\]
The \newword{conjugator length function} for~$G$ (with respect to~$S$) is the nondecreasing function $\mathrm{CLF}\colon\mathbb{N}\to\mathbb{N}$ defined by
\[
\CLF(n) = \max\{\cd(g,h) \mid g,h\in G\text{ are conjugate and }\ell(g)+\ell(h)\leq n\}.
\]
This definition first appeared in Andrew Sale's doctoral dissertation~\cite{Sale}, where it is credited to Tim Riley. Though this definition of $\CLF$ depends on the generating set, if $\CLF$ and $\CLF'$ are conjugator length functions corresponding to two different finite generating sets for~$G$ then there exists a constant $k>0$ so that
\[
\CLF'(n) \leq k\,\CLF\bigl(\lfloor kn\rfloor \bigr)
\]
for all $n$.  In particular, if $\CLF$ has polynomial growth then the degree of the polynomial is independent of the generating set.  By similar reasoning, the degree of polynomial growth of $\CLF$ is a quasi-isometry invariant for finitely generated groups.

The function $\CLF$ can be viewed as measuring the difficulty of the conjugacy problem in~$G$.  If $G$ has solvable word problem, then the conjugacy problem is solvable in $G$ if and only if $\CLF$ is a computable function, or equivalently if and only if $\CLF$ has a computable upper bound.  
The conjugator length function has been
estimated for many classes of groups. It has
been shown to be linear in hyperbolic groups
\cite{Lysenok}, mapping class groups \cite{Behrstock-Drutu,Masur-Minsky,Tao} and some metabelian groups
~\cite{Sale1,Sale2}, including lamplighter groups $\mathbb{Z}_q\wr \mathbb{Z}$ and solvable Baumslag-Solitar groups. It is
at most quadratic in fundamental groups of prime 3-manifolds~\cite{Sale1}, at most cubic in free solvable groups~\cite{Sale3}, and
at most exponential in CAT(0)-groups~\cite{Bridson-Haefliger}
and in certain semidirect products
$\mathbb{Z}^d \rtimes \mathbb{Z}^k$~\cite{Sale1}. 
Sale also gives examples in~\cite{Sale3} of wreath products whose conjugator length functions have quadratic lower bounds.
In upcoming work~\cite{BRS}, Bridson, Riley,
and Sale give examples of finitely presented groups whose conjugator length function is polynomial of arbitrary degree, as well an example of a finitely presented group whose conjugator length function grows like~$2^n$.

Thompson's group $F$ is the group defined by the presentation
\[
\langle x_0,x_1,x_2. \ldots \mid x_nx_k=x_kx_{n+1}\text{ for }n>k\rangle.
\]
This is one of three groups introduced by Richard J.~Thompson in the 1960's, which have since become important examples in geometric group theory.  See \cite{CFP} for a general introduction to Thompson's groups.  Since $x_n=x_0^{1-n}x_1x_0^{n-1}$ for all $n\geq 2$, Thompson's group $F$ is generated by the elements $\{x_0,x_1\}$.  In fact there is a presentation for $F$ with these generators and two relations (see~\cite{CFP}).

 Our main theorem is the following.

\begin{maintheorem}
The conjugator length function for Thompson's group $F$ has quadratic growth.  That is, there exist constants $0<a<b$ so that
\[
an^2 \leq \CLF(n) \leq bn^2
\]
for all sufficiently large~$n$.
\end{maintheorem}

The conjugacy problem for $F$ was first solved by V.~Guba and M.~Sapir as a special case of the solution for diagram groups~\cite{GuSa}.  In~\cite{BeMa}, the authors gave a different description of this solution using the language of strand diagrams, and in~\cite{BHMM1,BHMM2} it was shown that the solution to the conjugacy problem could be implemented in linear time.  All of our work here is phrased using strand diagrams, but our proof of the upper bound in Section~2 essentially follows Guba and Sapir's proof \mbox{\cite[Theorem~15.23]{GuSa}} while keeping track of the lengths of the conjugators.

We prove the lower bound by exhibiting a sequence of pairs of conjugate elements $(f_n,g_n)$ whose lengths grow linearly with $n$ but whose conjugator distance grows quadratically.  The main idea is that the ``area'' of a conjugating strand diagram can be forced to be much larger than areas of the strand diagrams for the two conjugate elements, as shown in Figure~\ref{fig:BigPicture}.  The elements we choose have cyclic centralizers, which makes it easy to compute an explicit lower bound for the conjugator distance.

All of the arguments in this paper can be modified to work for the generalized Thompson groups $F_n$ (see~\cite{Bro}), and more generally for diagram groups over finite presentations of finite semigroups (see~\cite{GuSa}).

We prove that a quadratic lower bound holds for $T$ and $V$ as well (see Theorem~\ref{lower}), and we would conjecture that a quadratic upper bound holds for $T$ as well using a modified version of the arguments in Section~2. In Thompson's group $V$ the word length is not comparable to the complexity of a strand diagram (see~Remark~\ref{rem:NotV}), so
the methods in Section~2 cannot be modified to obtain an upper bound better than $\CLF(n) \leq C (n\log n)^2$ for some constant~$C$.

This paper is organized as follows.  In Section~1 we establish a linear relationship between the word length of an element and the number of nodes in the corresponding strand diagram.  In Section~2 we prove a quadratic upper bound for the conjugator length function using the known solution to the conjugacy problem in~$F$.  Finally, in Section~3 we prove a quadratic lower bound for the conjugator length function by exhibiting the aforementioned sequence of pairs~$(f_n,g_n)$ and analyzing their length and conjugator distance.

\subsection*{Acknowledgements}
The authors would like to thank Timothy Riley and Andrew Sale for drawing our attention to the question addressed in this paper, and for suggesting the inclusion of Theorem~\ref{lower}.  We would also like to thank Collin Bleak as well as an anonymous referee for several helpful suggestions and comments.

\section{Strand Diagrams}

Here we briefly recall the definition of strand diagrams for Thompson's group~$F$ and the associated solution to the conjugacy problem given in~\cite{BeMa}.

\begin{figure}
\centering
\includegraphics{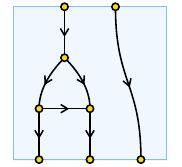}
\caption{A $(2,3)$-strand diagram.  Here $[0,1]\times[0,1]$ is shown as a blue square, but usually the square is not explicitly shown.}
\label{fig:FStrandDiagram32}
\end{figure}
A \newword{strand diagram} (see Figure~\ref{fig:FStrandDiagram32}) is a finite acyclic digraph embedded in the unit square $[0,1]\times[0,1]$, with the following properties:
\begin{enumerate}
\item The graph has finitely many univalent sources along the top edge of the square, and finitely many univalent sinks along the bottom edge of the square.\smallskip
\item Every other vertex is trivalent, and is either a \newword{split} (with one incoming edge and two outgoing edges) or a \newword{merge} (with two incoming edges and one outgoing edge).
\end{enumerate}
By convention, isotopic strand diagrams are considered equal.  A strand diagram $\f$ with $i$ sources and $j$ sinks will be referred to as an \newword{$\boldsymbol{(i,j)}$-strand diagram}.

\begin{figure}
\centering
\includegraphics{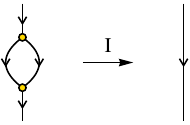}
\qquad\qquad\qquad\qquad
\includegraphics{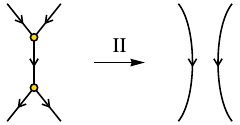}
\caption{Reductions of types I and II for strand diagrams (picture taken from~\cite{BeMa}).}
\label{fig:Reductions}
\end{figure}
A \newword{reduction} of a strand diagram is either of the two moves shown in Figure~\ref{fig:Reductions}.  A strand diagram is \newword{reduced} if it is not subject to any reductions.  Two strand diagrams are \newword{equivalent} if one can be obtained from the other by a sequence of reductions and inverse reductions. It is easy to show that every strand diagram is equivalent to a unique reduced strand diagram.

\begin{figure}
    \centering
    $\begin{array}{cc}
         \f & \raisebox{-0.47\height}{\includegraphics{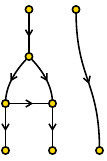}} \\[0.55in]
         \g & \raisebox{-0.47\height}{\includegraphics{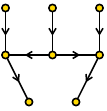}}
    \end{array}\qquad\qquad
    \underset{\textstyle\f\cdot\g\rule{0pt}{2.2ex}}{\raisebox{-0.47\height}{\includegraphics{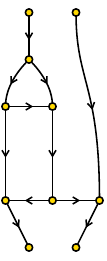}}}\qquad\qquad
    \underset{\textstyle\f\g\rule{0pt}{2.2ex}}{\raisebox{-0.47\height}{\includegraphics{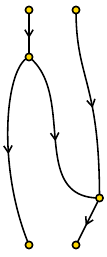}}}$
    \caption{Two strand diagrams $\f$ and $\g$, their concatenation~$\f\cdot\g$, and their product~$\f\g$.  In this case, $\f\g$ is obtained from $\f\cdot\g$ by two reductions, with the first of type~II and the second of type~I.}
\label{fig:ProductStrandDiagrams}
\end{figure}
If $\f$ is an $(i,j)$-strand diagram and $\g$ is a $(j,k)$-strand diagram, the \newword{concatenation} $\f\cdot\g$ is the strand diagram obtained by gluing the sinks of $\f$ to the sources of~$\g$ and then removing the resulting bivalent vertices~(see Figure~\ref{fig:ProductStrandDiagrams}). The \newword{inverse} $\f^{-1}$ of an $(i,j)$-strand diagram $\f$ is the $(j,i)$-strand diagram obtained by reflecting $\f$ along a horizontal line.  Note that the concatenations $\f\cdot\f^{-1}$ and $\f^{-1}\cdot\f$ can both be reduced to \newword{trivial} strand diagrams, i.e.\ strand diagrams that have no splits or merges.

If $\f$ is a reduced $(i,j)$-strand diagram and $\g$ is a reduced $(j,k)$-strand diagram, the \newword{product} $\f\g$ is the reduced strand diagram obtained by reducing the concatenation~\mbox{$\f\cdot\g$} (see Figure~\ref{fig:ProductStrandDiagrams}).  Under this product operation, the set of all reduced strand diagrams forms a groupoid (i.e.~category with inverses) whose objects are the positive integers and whose morphisms are reduced strand diagrams.  

\begin{figure}
\centering
\includegraphics{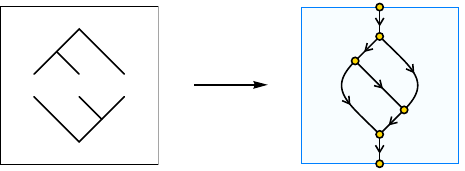}
\caption{Constructing a $(1,1)$-strand diagram from a tree-pair diagram (picture taken from~\cite{BeMa}).}
\label{fig:TreeToStrand}
\end{figure}
For the purposes of this paper, \newword{Thompson's group $\boldsymbol{F}$} will be viewed as the group of all reduced $(1,1)$-strand diagrams.  We will use Roman letters ($f$ and $g$) instead of Fraktur letters ($\f$ and~$\g$) when referring to elements of~$F$.  As a group, $F$ is generated by the two elements $x_0$ and $x_1$ shown in Figure~\ref{fig:GeneratorStrandDiagrams}.
\begin{figure}
    \centering
    $\underset{\textstyle\rule{0pt}{10pt} x_0}{\includegraphics{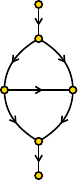}}
    \qquad\qquad\qquad
    \underset{\textstyle\rule{0pt}{10pt} x_1}{\includegraphics{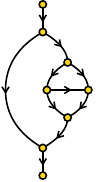}}$
    \caption{Strand diagrams for the generators $\{x_0,x_1\}$ of~$F$.}
    \label{fig:GeneratorStrandDiagrams}
\end{figure}%
This description of~$F$ using strand diagrams is closely related to the usual description of $F$ using tree-pair diagrams as given in~\cite{CFP}.  Specifically, given any reduced tree-pair diagram for an element of $F$, we can construct the corresponding reduced $(1,1)$-strand diagram by gluing together the leaves of the two trees, as shown in Figure~\ref{fig:TreeToStrand}.

We will need a few more definitions involving strand diagrams that do not appear in~\cite{BeMa}.

\begin{definition}\quad
\begin{enumerate}
    \item If $\f$ and $\g$ are reduced strand diagrams for which the product $\f\g$ is defined, then there exist unique reduced strand diagrams $\f'$, $\g'$, and $\h$ so that
    \[
    \f= \f'\cdot\h,\qquad \g=\h^{-1}\cdot\g'\qquad\text{and}\qquad \f\g\ = \f'\cdot \g'.
    \]
    In this case, we say that the product $\f\g$ is obtained by \newword{canceling}~$\h$.\smallskip
    \item If $\f$ is an $(i,j)$-strand diagram and $\g$ is an $(i',j')$-strand diagram, we let $\f\oplus\g$ denote the $(i+i',j+j')$-strand diagram obtained by placing $\g$ to the right of~$\f$.\smallskip
    \item For each positive integer $k$, the  \newword{right vine} with $k$ leaves is the $(1,k)$-strand diagram $\t_k$ shown in Figure~\ref{fig:RightVine}.  If $\f$ is any reduced $(i,j)$-strand diagram, then the product $\t_i\f\t_j^{-1}$ is an element of~$F$.
\end{enumerate}
\end{definition}
\begin{figure}
    \centering
    \includegraphics{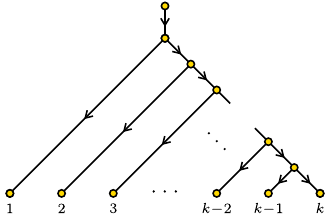}
    \caption{The right vine with $k$ leaves, denoted~$\t_k$.}
    \label{fig:RightVine}
\end{figure}

Finally we recall from \cite{BeMa} the solution to the conjugacy problem in $F$ using strand diagrams, which was based on the solution to the conjugacy problem given by Guba and Sapir~\cite{GuSa}.  An \newword{annular strand diagram} is a finite digraph embedded in the annulus $[0,1]\times S^1$, with the following properties:
\begin{enumerate}
\item Every vertex is either a split or a merge.\smallskip
\item Every directed cycle winds counterclockwise around the central hole.\smallskip
\item Some edges may be \newword{free loops} without any vertices, which must wind counterclockwise around the central hole.
\end{enumerate}
As with strand diagrams, isotopic annular strand diagrams are considered equal.
\begin{figure}
\centering
\includegraphics{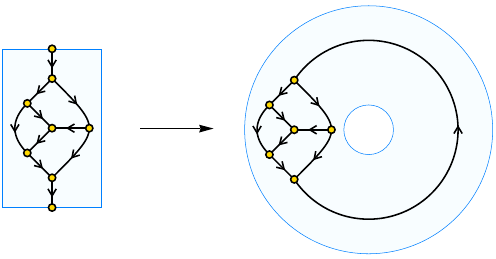}
\caption{Closing a strand diagram to obtain an annular strand diagram (picture taken from~\cite{BeMa}).}
\label{fig:ClosingStrandDiagram}
\end{figure}%
If~$\f$ is any $(i,i)$-strand diagram, its \newword{closure} is the annular strand diagram obtained by gluing its sources and sinks together and removing the resulting bivalent vertices, as shown in Figure~\ref{fig:ClosingStrandDiagram}.

\begin{figure}
    \centering
    \fbox{$\;\;
    \raisebox{-0.47\height}{\includegraphics{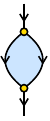}}
    \;\;\overset{\textstyle\text{I}}{\xrightarrow{\quad}}\;\;
    \raisebox{-0.47\height}{\includegraphics{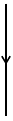}}
    \;\;\;$}
    \hfill
    \fbox{$\;\;
    \raisebox{-0.47\height}{\includegraphics{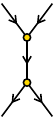}}
    \;\;\overset{\textstyle\text{II}}{\xrightarrow{\quad}}\;\;
    \raisebox{-0.47\height}{\includegraphics{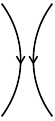}}
    \;\;$}
    \hfill
    \fbox{$
    \raisebox{-0.47\height}{\includegraphics{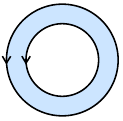}}
    \;\;\overset{\textstyle\text{III}}{\xrightarrow{\quad}}\;\;
    \raisebox{-0.47\height}{\includegraphics{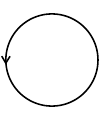}}
    $}
    \caption{Reductions of type I, II or III for annular strand diagrams.  In the first move, the shaded disk must not contain the central hole.  In the third move, both loops must be free loops, and the shaded annulus must not contain any vertices. (Picture taken from~\cite{BeMa}.)}
    \label{fig:reduction-types}
\end{figure}%
A \newword{reduction} of an annular strand diagram is any one of the moves shown in Figure~\ref{fig:reduction-types}.  An annular strand diagram is \newword{reduced} if it is not sibject to any reductions.  Two annular strand diagrams are \newword{equivalent} if one can be obtained from the other by a sequence of reductions and inverse reductions.  Again, every annular strand diagram is equivalent to a unique reduced annular strand diagram.  

The following theorem is proven in~\cite[Section~3]{BeMa}.  In the case where $i=j=1$ it gives a solution to the conjugacy problem in Thompson's group~$F$.  Our ideas in Sections~\ref{sec:UpperBound} and \ref{sec:LowerBound} are both based on the proof of this theorem.

\begin{theorem}\label{thm:ConjugacySolution}
Let\/ $\f$ be a reduced\/ $(i,i)$-strand diagram and let\/ $\g$ be a reduced\/ $(j,j)$-strand diagram.  Then the following are equivalent:
\begin{enumerate}
    \item The elements\/ $\f$ and\/ $\g$ are conjugate in the groupoid of reduced strand diagrams.\smallskip
    \item The reduced annular strand diagrams obtained by closing\/ $\f$ and\/ $\g$ and reducing are the same.
\end{enumerate}
\end{theorem}
\begin{proof}[Sketch of Proof]
The implication (1) $\Rightarrow$ (2) is easy.  For the converse, call a strand diagram \newword{cyclically reduced} if its closure is already reduced as an annular strand diagram.  It is not hard to show that every $(k,k)$-strand diagram is conjugate to a cyclically reduced strand diagram (see~\cite[Proposition~3.2]{BeMa} or the stronger Lemma~\ref{lem:MainLemma1} below), so we may assume that $\f$ and $\g$ are cyclically reduced and have the same closure.

Let $\f^\infty$ be the lift of the closure of $\f$ to the universal cover of the annulus.  Then $\f^\infty$ can be viewed as an infinite concatenation of $\f$'s, i.e.
\[
\f^\infty = \bigcup_{k\in\mathbb{Z}} \f_k
\]
where each $\f_k$ is a copy of~$\f$ and the sinks of each $\f_k$ are the same as the sources of~$\f_{k+1}$.  We can also decompose $\f^\infty$ as an infinite concatenation $\bigcup_{k\in\mathbb{Z}} \g_k$ of copies of~$\g$.  Indeed, we can choose such a decomposition so that $\g_0\subseteq \bigcup_{k=1}^\infty \f_k$.
\begin{figure}
    \centering
    \includegraphics{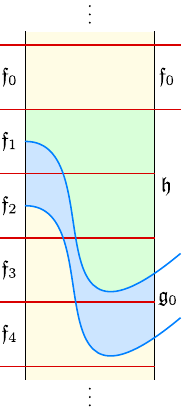}
    \caption{The deck transformation $\delta$ takes $\f_0\cup \h$ to $\h\cup \g_0$, and hence $\f\h=\h\g$.}
    \label{fig:UniversalCover}
\end{figure}
Let $\h$ be the strand diagram that lies between the bottom of $\f_0$ and the top of~$\g_0$.
Then $\delta(\f_0\cup \h) = \h\cup \g_0$, where $\delta\colon \f^\infty\to \f^\infty$ is the deck transformation that maps each $\f_k$ to $\f_{k+1}$ (see~Figure~\ref{fig:UniversalCover}).  We conclude that the strand diagrams $\f\h$ and $\h\g$ are the same, so $\f$ and $\g$ are conjugate in the groupoid of strand diagrams.
\end{proof}

\section{Norm and Length}

For $f\in F$, let $\ell(f)$ denote the word length of $f$ with respect to the $\{x_0,x_1\}$ generating set. An explicit formula for $\ell(f)$ was first given by Fordham~\cite{For}, and variants of Fordham's formula were subsequently published by Belk and Brown~\cite{BeBr} and Guba~\cite{Guba}.

Define the \newword{norm} $\|\f\|$ of a strand diagram $\f$ is its number of interior nodes (i.e.~merges and splits).  Note that $\|\f^{-1}\|=\|\f\|$, $\|\f\cdot \g\|=\|\f\|+\|\g\|$, $|\f\oplus\g\| = \|\f\| + \|\g\|$, and $\|\f\g\| \leq \|\f\|+\|\g\|$ for all strand diagrams $\f$ and $\g$.

If $f$ is an element of $F$ (i.e.~a reduced $(1,1)$-strand diagram), then $f$ always has the same number of merges as splits, and therefore the norm $\|f\|$ must be even. The following proposition relates the norm of each element of $F$ to its word length.

\begin{proposition}\label{prop:NormAndLength}
For any $f\in F$, we have
\[
\frac{\|f\|}{2}-2 \leq \ell(f) \leq 2\hspace{0.083333em}\|f\|.
\]
\end{proposition}
\begin{proof}
Observe that each tree in the reduced tree pair diagram for $f$ has $\|f\|/2$ carets. In \cite{For}, Fordham gives a formula for the length of an element as the sum of weights assigned to  corresponding pairs of carets in a tree pair diagram.  Each of his weights is at most~$4$, so it follows easily that $\ell(f)  \leq 2\hspace{0.0833333em}\|f\|$.

For the lower bound, observe that $\|\t_3^{-1}x_0\t_3\|=\|\t_3^{-1}x_1\t_3\| = 2$, where $\t_3$ is the right vine with $3$~leaves.  Given a word $f = s_1^{\epsilon_1}\cdots s_n^{\epsilon_n}$ where each $s_i\in\{x_0,x_1\}$ and each $\epsilon_i=\pm 1$, we can write 
\[
f = \t_3\bigl(\t_3^{-1}s_1\t_3\bigr)^{\epsilon_1}\cdots \bigl(\t_3^{-1}s_n\t_3\bigr)^{\epsilon_n}\t_3^{-1}
\]
and hence
\[
\|f\| \leq \|\t_3\| + \sum_{i=1}^n \|\t_3^{-1}s_i\t_3\| + \|\t_3\| = 2n+4.
\]
Thus $\|f\|\leq 2\hspace{0.08333em}\ell(f)+4$, so $\ell(f) \geq \|f\|/2 - 2$.
\end{proof}

\begin{remark}\label{rem:BetterNormLength}
In fact we have \[
\frac{\|f\|}{2} - 2 \leq \ell(f) \leq 2\hspace{0.083333em}\|f\|-8
\]
whenever $\|f\|\geq 6$, since the leftmost pair of corresponding carets in a tree pair diagram always has weight~$0$ and the two rightmost pairs of corresponding carets each have weight at most~$2$. Both bounds are sharp, with the lower bound realized by the elements $x_1^n$ and the upper bound realized by the elements
\[
x_0^2\bigl(x_1x_0^{-1}\bigr)^nx_1^{-1}\bigl(x_0x_1^{-1}\bigr)^nx_0^{-1}.
\]
\end{remark}

\begin{remark}\label{rem:NotV}
Burillo, Cleary, Stein, and Taback have proven an analog of Proposition~\ref{prop:NormAndLength} for Thompson's group~$T$~\mbox{\cite[Theorem~5.1]{BCST}}, but no analogous result holds for Thompson's group $V$.  The trouble is that $V$ allows arbitrary permutations of the leaves of a tree diagram, so there are at least $(n+1)!$ different elements $f\in V$ with $\|f\|\leq 2n$, and therefore $\ell(f)$ is not bounded above by any linear function of~$\|f\|$.  However, Birget has proven that there exists a constant $C>0$ such that $\ell(f) \leq C\hspace{0.08333em} \|f\| \log \|f\|$ for all $f\in V$~\cite[Theorem~3.8]{Birget}.
\end{remark}

\section{An Upper Bound}
\label{sec:UpperBound}

In this section we prove our upper bound for the conjugator length in ~$F$.  Throughout this section, we say that an $(i,i)$-strand diagram $\f$ is \newword{strongly cyclically reduced} if its closure is already a reduced annular strand diagram, and this has the same number of connected components as~$\f$.  (These correspond to the ``absolutely reduced normal diagrams'' defined by Guba and Sapir in~\cite{GuSa}.)

\begin{lemma}\label{lem:MainLemma1}
Let\/ $\f$ be a nontrivial reduced\/ $(i,i)$-strand diagram with\/ $\|\f\|=n$.  Then there exists a reduced\/ $(i,j)$-strand diagram\/ $\h$ so that\/ $\f'=\h^{-1}\f\h$ is strongly cyclically reduced,\/ $\|\f'\|\leq n$, and
\[
\|\h\|\leq 1+\frac{n(n+4i-6)}{8}.
\]
\end{lemma}
\begin{proof}
We proceed by induction on $i+n$.  The base case is $i+n=1$, for which $\f$ is the trivial $(1,1)$-strand diagram and is therefore already strongly cyclically reduced.  

For the induction step, suppose first that the closure of $\f$ has fewer components than~$\f$. This occurs when $\f$ can be written as $\f = \f_1\oplus \f_2$, 
where $\f_1$ is a $(j,k)$-strand diagram with $j\ne k$ and $\f_2$ is an $(i-j,i-k)$-strand diagram.  Without loss of generality, suppose that $j<k$.  Then we can rewrite $\f$ as a product
\[
\f = (\e_j \oplus \f_2)(\f_1 \oplus \e_{i-k})
\]
as shown in Figure~\ref{fig:DiagonalCut},
\begin{figure}
    \centering
    $\raisebox{-0.47\height}{\includegraphics{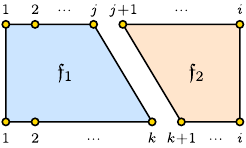}}\;\;\quad=\quad\;\;\raisebox{-0.47\height}{\includegraphics{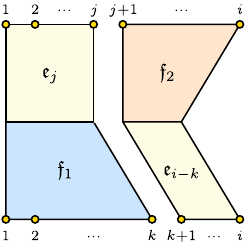}}$
    \caption{The equality $\f_1\oplus \f_2 = (\e_j\oplus \f_2)(\f_1\oplus\e_{i-k})$, where $\e_j$ and $\e_{i-k}$ are trivial strand diagrams.}
    \label{fig:DiagonalCut}
\end{figure}%
where $\e_j$ and $\e_{i-k}$ denote trivial strand diagrams with $j$ strands and $i-k$ strands, respectively.  Let $\k$ be whichever of $(\e_j\oplus\f_2)$ and $(\f_1\oplus \e_{i-k})^{-1}$ has fewer interior nodes, and let $\f'=\k^{-1}\f\k=(\f_1\oplus \e_{i-k})(\e_j\oplus \f_2)$. Then $\f'$ is an $(i+j-k,i+j-k)$-strand diagram with $\|\f'\|\leq n$ (there may be fewer than $n$ interior nodes if $\f'$ is not initially reduced), and $\|\k\|\leq n/2$.  Since $i+j-k\leq i-1$, we know that $(i+j-k)+\|\f'\|< i+n$. Therefore, it follows from our induction hypothesis that there exists an $\h$ with 
\[
\|\h\|\leq 1 + \frac{n(n+4(i-1)-6)}{8} = 1+ \frac{n(n+4i-10)}{8}.
\]
such that $\h^{-1}\f'\h$ is strongly cyclically reduced.  Then $(\k\h)^{-1}\f(\k\h)$ is strongly cyclically reduced and
\[
\|\k\h\|\leq \|\k\|+\|\h\| \leq \frac{n}{2} + 1 + \frac{n(n+4i-10)}{8} = 1+\frac{n(n+4i-6)}{8}.
\]

Now consider the case where the closure of $\f$ has the same number of components as~$\f$.  Note then that any trivial strands of $\f$ (i.e.~edges whose endpoints are a source and a sink) must correspond to free loops in the closure. If the closure of $\f$ is reduced then we are done, so it must be possible to apply a reduction of type I, II, or III to the closure of~$\f$ as described in~\cite{BeMa} (see Figure~\ref{fig:reduction-types}).

Suppose first that the closure of $\f$ is subject to a type~I reduction.  Since all the trivial strands of $\f$ correspond to free loops, there must be a $1\leq j\leq i-1$ so that sources $j$ and $j+1$ are connected to a merge and sinks $j$ and $j+1$ are connected to a split.  Let $\k$ be the $(i,i-1)$-strand diagram with exactly one merge connected to sources $j$ and $j+1$ and sink $j$.  Then $\f'=\k^{-1}\f\k$ is a nontrivial $(i-1,i-1)$-strand diagram with $\|\f'\|= n-2$.  Since $(i-1)+(n-2) < i+n$, our induction hypothesis tells us that there exists a reduced strand diagram $\h$ with
\[
\|\h\| \leq 1 + \frac{(n-2)\bigl((n-2)+4(i-1)-6\bigr)}{8} = 1 + \frac{(n-2)(n+4i-12)}{8}
\]
such that $\h^{-1}\f'\h$ is strongly cyclically reduced.  Then $(\k\h)^{-1}\f(\k\h)$ is strongly cyclically reduced and
\[
\|\k\h\|\leq \|\k\|+\|\h\| \leq 1 + 1 + \frac{(n-2)(n+4i-12)}{8} \leq 1 + \frac{n(n+4i-6)}{8},
\]
where the last inequality follows from the fact that $n\geq 2$ and $i\geq 2$.

Now suppose that the closure of $\f$ is subject to a type~II reduction.  Again, since all the trivial strands of $\f$ correspond to free loops, there must exist a $1\leq j\leq i$ so that source $j$ of $\f$ is connected to a split and sink $j$ of $\f$ is connected to a merge.  Let $\k$ be the $(i,i+1)$-strand diagram with exactly one split connected to source $j$ and sinks $j$ and $j+1$.  Then $\f'=\k^{-1}\f\k$ is a nontrivial $(i+1,i+1)$-strand diagram with $\|\f'\|= n-2$.  Since $(i+1)+(n-2)<i+n$, our induction hypothesis tells us that there exists a reduced strand diagram $\h$ with
\[
\|\h\|\leq 1+ \frac{(n-2)\bigl((n-2)+4(i+1)-6\bigr)}{8} = 1+\frac{(n-2)(n+4i-4)}{8}
\]
so that $\h^{-1}\f'\h$ is strongly cyclically reduced.  Then $(\k\h)^{-1}\f(\k\h)$ is strongly cyclically reduced and
\[
\|\k\h\|\leq \|\k\|+\|\h\| \leq 1 + 1 + \frac{(n-2)(n+4i-4)}{8} \leq 1 + \frac{n(n+4i-6)}{8},
\]
where the last inequality follows from the fact that $i\geq 2$.

Finally, suppose that the closure of $\f$ is subject to a type~III reduction.  Then there exists a $1\leq j\leq i-1$ so that source $j$ is connected directly to sink $j$ and source $j+1$ is connected directly to sink $j+1$ in~$\f$.  Let $\k$ be the $(i,i-1)$-strand diagram with a single merge connected to sources $j$ and $j+1$ and sink $j$. Then $\f'=\k^{-1}\f\k$ is an $(i-1,i-1)$-strand diagram with $\|f'\|=n$.  Since $(i-1)+n< i+n$, our induction  hypothesis tells us that there exists a reduced strand diagram $\h$ with
\[
\|\h\|\leq 1+ \frac{n\bigl(n+4(i-1)-6\bigr)}{8} = 1+\frac{n(n+4i-10)}{8}
\]
so that $\h^{-1}\f'\h$ is strongly cyclically reduced.  Then $(\k\h)^{-1}\f(\k\h)$ is strongly cyclically reduced and
\[
\|\k\h\|\leq \|\k\|+\|\h\| \leq 1 + 1 + \frac{n(n+4i-10)}{8} \leq 1 + \frac{n(n+4i-6)}{8}.
\]
The last inequality follows from the fact that $\f$ is nontrivial, and hence $n\geq 2$.
\end{proof}

\begin{corollary}\label{cor:ElementFStronglyCyclicallyReduced}
Let $f\in F$ with $\|f\|=n$.  Then there exists a reduced\/ $(1,j)$-strand diagram\/ $\h$ so that\/ $\f'=\h^{-1}\f\h$ is strongly cyclically reduced,\/ $\|\f'\|\leq n$, and
\[
\|\h\| \leq \frac{(n-1)^2 +7}{8}.
\]
\end{corollary}

\begin{lemma}\label{lem:SameClosure}
Let\/ $\f$ and\/ $\g$ be strongly cyclically reduced strand diagrams whose closures are the same, and let $n=\|\f\|=\|\g\|$.  Then there exists a reduced strand diagram\/~$\h$ so that\/ $\g=\h^{-1}\f\h$ and\/ $\|\h\|\leq \frac{3}{2}n^2$.
\end{lemma}
\begin{proof}
Suppose first that $\f$ and $\g$ are connected.  If $\f$ and $\g$ are the identity we are done, so suppose $\f$ and $\g$ are nontrivial.  As in the sketch of the proof of Theorem~\ref{thm:ConjugacySolution}, let $\f^\infty$ be the lift of the closure of $\f$ to universal cover of the annulus, with
\[
\f^\infty = \bigcup_{k\in\mathbb{Z}} \f_k = \bigcup_{k\in\mathbb{Z}} \g_k.
\]
Note that we can choose the decomposition into $\g_k$'s so that $\g_0\subseteq \bigcup_{k=1}^\infty \f_k$ and $\g_0$ intersects $\f_1$.
As before, we have $\f\h = \h\g$, where $\h$ is the strand diagram that lies between $\f_0$ and~$\g_0$ (see~Figure~\ref{fig:UniversalCover}).  We claim that
\[
\g_0 \subseteq \bigcup_{k=1}^{3n/2} \f_k.
\]
It follows that $\h\subseteq \bigcup_{k=1}^{3n/2} \f_k$, so $\|\h\|\leq (3n/2)\|\f\| = 3n^2/2$.

To prove the claim, define the \textit{full edges} of $\f_0$ to be those that start and end at trivalent vertices, and the \textit{half edges} of $\f_0$ to be those that have either a source or a sink at one end.  (Since $\f_0$ is connected and nontrivial, there are no edges directly from a source to a sink.) We place a $\delta$-invariant geodesic metric on $\f^\infty$ so that each full edge of $\f_0$ has length~$1$ and each half edge has length~$1/2$.  Since $\f_0$ has exactly $n$ interior nodes, the total length of all of the edges of $\f_0$ is~$3n/2$. Then the total length of all of the edges of $\g_0$ must be exactly $3n/2$, and in particular the diameter of $\g_0$ is at most $3n/2$. Since the minimum distance from a source to a sink in $\f_0$ is at least~$1$, the claim follows easily.

For the general case, suppose $\f=\f_1\oplus \cdots \oplus \f_k$, where each $\f_i$ is connected and has the same number of sources as sinks.  Since $\f$ and $\g$ are strongly cyclically reduced and have the same closure, it follows that $\g=\g_1\oplus \cdots \oplus \g_k$, where each $\g_i$ is connected, has the same number of sources as sinks, and has the same closure as~$\f_i$.  By the argument above, there exists for each $i$ a reduced strand diagram $\h_i$ with $\|\h_i\|\leq 3\|\f_i\|^2/2$ such that $\g_i=\h_i^{-1}\f_i\h_i$.  Then $\g=\h^{-1}\f\h$, where $\h=\h_1\oplus \cdots \oplus \h_n$ and
\[
\|\h\|=\sum_{i=1}^k \|\h_i\| \leq \sum_{i=1}^k \frac{3}{2}\|\f_i\|^2 \leq \frac{3}{2}\biggl(\sum_{i=1}^k \|\f_i\|\biggr)^2 = \frac{3}{2}\|\f\|^2.\qedhere
\]
\end{proof}

\begin{theorem}\label{thm:ConjugacyLength}
Let $f,g\in F$, with $\|f\|=i$ and $\|g\|=j$. Suppose $f$ and $g$ are conjugate, with the corresponding reduced annular strand diagram having $k$ nodes.  Then there exists an $h\in F$ so that $g=h^{-1}fh$ and
\[
\|h\| \leq \frac{(i-1)^2 + (j-1)^2 + 12k^2 + 14}{8}.
\]
\end{theorem}
\begin{proof}
By Corollary~\ref{cor:ElementFStronglyCyclicallyReduced}, there exist reduced strand diagrams $\h_1$ and $\h_2$ with
\[
\|\h_1\| \leq \frac{(i-1)^2+7}{8}\qquad\text{and}\qquad \|\h_2\|\leq \frac{(j-1)^2+7}{8}
\]
so that $\f=\h_1^{-1}f\h_1$ and $\g=\h_2^{-1}g\h_2$ are strongly cyclically reduced.  Then $\f$ and~$\g$ have the same closure and $\|\f\|=\|\g\|=i$, so by Lemma~\ref{lem:SameClosure} there exists a strand diagram $\h_3$ with $\|\h_3\|\leq \frac{3}{2}k^2$ so that $\g=\h_3^{-1}\f\h_3$.  Let $h=\h_1\h_3\h_2^{-1}$.  Then $g=h^{-1}fh$ which means that $h\in F$, and
\[
\|h\|\leq \|\h_1\|+\|\h_2\|+\|\h_3\| \leq \frac{(i-1)^2 + (j-1)^2 + 12k^2 + 14}{8}.\qedhere
\]
\end{proof}

\begin{corollary}\label{cor:MainCorollary}
Let $f$ and $g$ be conjugate elements of $F$ with $\ell(f)=i$ and $\ell(g) =j$, where $i\leq j$.  Then
\[
\cd(f,g) \leq 13i^2+j^2+27i+3j+20.
\]
In particular, the conjugator length function of Thompson's group $F$ satisfies
\[
\mathrm{CLF}(n)\leq \frac{7}{2}n^2+15n+20
\]
for all $n\in\mathbb{N}$.
\end{corollary}
\begin{proof}
By Proposition~\ref{prop:NormAndLength}, we know that $\|f\|\leq 2i+4$ and $\|g\|\leq 2j+4$.  Moreover, since the closure of a nontrivial $(1,1$)-strand diagram is always subject to at least one reduction, the reduced annular strand diagram for $f$ (and hence $g$) has at most $\|f\|-2 = 2i+2$ nodes.  By Theorem~\ref{thm:ConjugacyLength} and Proposition~\ref{prop:NormAndLength}, there exists an $h\in F$ so that $g=h^{-1}fh$  and
\begin{multline*}
\qquad \ell(h) \leq 2\hspace{0.08333em}\|h\| \leq 
2\hspace{0.083333em}\frac{(2i+3)^2 + (2j+3)^2 + 12(2i+2)^2 + 14}{8} \\[3pt]
=
13i^2+j^2+27i+3j+20.
\qquad
\end{multline*}
If $n=i+j$, it follows that
\begin{align*}
\ell(h) &\leq 13i^2+j^2+27i+3j+20 \\[3pt] &= \frac{7}{2}n^2+15n+20 - \frac{1}{2}(19i+5j+24)(j-i).
\end{align*}
Since $i\leq j$, the term being subtracted on the right is positive, so we conclude that
\[
\ell(h) \leq \frac{7}{2}n^2+15n+20.\qedhere
\]
\end{proof}

\section{A Lower Bound}
\label{sec:LowerBound}

In this section we prove a quadratic lower bound on conjugator lengths for elements of~$F$.  The idea of the proof is to construct strand diagrams $\f$ and $\g$ with a linear number of vertices so that the corresponding conjugator $\h$ (see Figure~\ref{fig:UniversalCover}) has a quadratic number of vertices.  Our strategy is to use a regular grid of width $2n$ for the universal cover~$\f^\infty$ (see the sketch of the proof of Theorem~\ref{thm:ConjugacySolution}), and we choose $\f$ and $\g$ so that the strand diagram $\h$ between them is a large triangular section of the grid, as shown in Figure~\ref{fig:BigPicture}.

Unfortunately, a lower bound for the conjugator length requires understanding \textit{all} conjugators between a given pair of elements.  This will be the main source of complication in our proof, and will require some known results about centralizers in~$F$. For the following proposition, a \newword{proper root} of an element $f\in F$ is an element $g\in F$ such that $f=g^k$ for some $k\geq 2$.

\begin{proposition}\label{prop:Centralizers}
Let $f\in F$ and suppose that $f$ has no proper roots in $F$ and the reduced annular strand diagram for $f$ is connected.  Then the centralizer of $f$ in $F$ is the cyclic group $\langle f\rangle$ generated by~$f$.
\end{proposition}
\begin{proof}
Guba and Sapir compute centralizers for elements of diagram groups in \cite[Theorem~15.35]{GuSa}, and this follows easily from their proof.
\end{proof}

\begin{figure}
\centering
\includegraphics{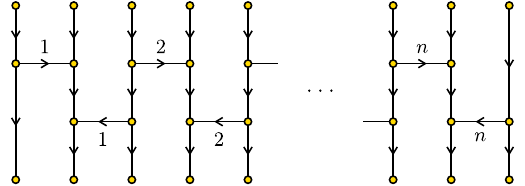}
\caption{The $(2n+1,2n+1)$-strand diagram~$\f_n$.}
\label{fig:Elementfn}
\end{figure}%
\begin{figure}
\centering
\includegraphics{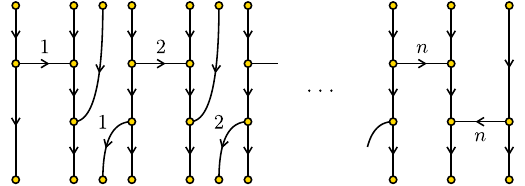}
\caption{The $(3n,3n)$-strand diagram~$\g_n$.}
\label{fig:Elementgn}
\end{figure}%
For any $n\geq 2$, let $\f_n$ and $\g_n$ be the strand diagrams shown in Figures~\ref{fig:Elementfn} and~\ref{fig:Elementgn}, and let $f_n,g_n\in F$ be the elements 
\[
f_n=\t_{2n+1}\f_n\t_{2n+1}^{-1}
\qquad\text{and}\qquad g_n=\t_{3n}\g_n\t_{3n}^{-1}
\]
where $\t_k$ denotes the right vine with $k$ leaves shown in Figure~\ref{fig:RightVine}. It is tedious but straightforward to check that
\begin{align*}
f_n &= x_0\, x_3^2\, x_7^2\, \cdots\, x_{4n-5}^2\, x_{4n-3}^{-2}\, \cdots\, x_5^{-2}\,x_1^{-2}\\[2pt]
&= x_0\bigl(x_1^{-1}x_0^{-1}x_1x_0^{-1}\bigr)^{n-1} x_1^{-2} \bigl(x_0x_1x_0x_1^{-1}\bigr)^{n-1}
\end{align*}
and
\begin{align*}
g_n &= x_0\,x_4^2\, x_9^2 \,\cdots\, x_{5n-6}^2\, x_{5n-4}^{-2} \,\cdots\, x_6^{-2}\, x_1^{-2} \\[2pt]
&= x_0\bigl(x_1^{-2}x_0^{-1}x_1x_0^{-1}\bigr)^{n-1}x_1^{-2}(x_0x_1x_0)^{n-1}.
\end{align*}
It follows that $\ell(f_n)\leq 8n-5$ and $\ell(g_n)\leq 8n-5$.

\begin{theorem}\label{thm:MainLowerBound}The elements $f_n$ and $g_n$ satisfy
\[
\cd(f_n,g_n) \geq \frac{n^2-5n-4}{2}.
\]
\end{theorem}
\begin{proof}
Since the closure of $\f_n$ is connected and already reduced, the reduced annular strand diagram for $f_n$ is connected.  Moreover, since $f_n'(0)=2$, the element $f_n$ has no proper roots in~$F$.   By Proposition~\ref{prop:Centralizers}, we deduce that the centralizer of $f_n$ is just $\langle f_n\rangle$.

\begin{figure}
\centering
\includegraphics{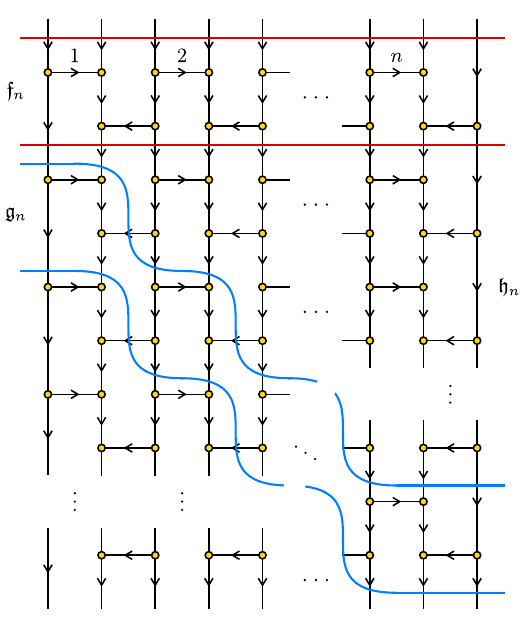}
\caption{The $(2n+1,3n)$-strand diagram $\h_n$ lies in the triangular region between the bottom red curve and the top blue curve. Note that $\f_n\h_n=\h_n\g_n$.}
\label{fig:BigPicture}
\end{figure}%
For each $1\leq i\leq n-1$, let
\[
\r_i = \e_{3i-1} \oplus \bigl(\f_{n-i}\cdot (\s \oplus \e_{2n-2i})\bigr)
\]
where $\e_j$ denotes the trivial $(j,j)$-strand diagram and $\s$ is the $(1,2)$-strand diagram with a single split.  Note that each $\r_i$ is a reduced $(2n+i,2n+i+1)$-strand diagram, with
\[
\|\r_i\| = \|\f_{n-i}\| + \|\s\| = 4(n-i)+1.
\]
Let
\[
\h_n = \r_1 \r_2 \cdots \r_{n-1}.
\]
Then $\h_n$ is the $(2n+1,3n)$-strand diagram shown in Figure~\ref{fig:BigPicture}, with each $\r_i$ containing two ``rows'' of interior nodes.  Note that
\[
\h_n = \sum_{i=1}^{n-1} \|\r_i\| = \sum_{i=1}^{n-1} \bigl(4(n-i)+1\bigr) = (n-1)(2n+1).
\]
Observe from the figure that $\f_n\h_n=\h_n\g_n$.  Then the element $h_n=\t_{2n+1}\h_n\t_{3n}^{-1}$ conjugates $f_n$ to $g_n$.  Since the centralizer of $f_n$ is $\langle f_n\rangle$, every conjugator from $f_n$ to $g_n$ must have the form $f_n^k h_n$ for some $k\in\mathbb{Z}$.  We must prove that $\ell(f_n^kh_n)\geq (n^2-5n-4)/2$ for every~$k\in\mathbb{Z}$.

It suffices to prove that $\bigl\|\f_n^k\h_n\bigr\| \geq n^2-1$ for each~$k$, since then
\begin{align*}
\bigl\|f_n^kh_n\bigr\| = \bigl\| \t_{2n+1}\f_n^k\h_n\t_{3n}^{-1}\bigr\|
&\geq \bigl\|\f_n^k\h_n\bigr\| - \|\t_{2n+1}\| - \|\t_{3n}\| \\
&\geq \bigl(n^2-1\bigr) - (2n) - (3n-1) = n(n-5)
\end{align*}
and therefore $\ell(f_n^kh_n) \geq n(n-5)/2-2 = (n^2-5n-4)/2$ by Proposition~\ref{prop:NormAndLength}. To compute~$\bigl\|\f_n^k\h_n\bigr\|$, observe that the concatenation $\f_n^k\cdot \h_n$ is not necessarily reduced, so we must worry about cancellation in the product $\f_n^k\h_n$.
There are three cases:
\begin{itemize}
    \item For $k\geq 0$, the concatenation $\f_n^k\cdot \h_n$ is reduced, so
    \[
    \bigl\|\f_n^k\h_n\bigr\|= k\|\f_n\|+\|\h_n\| \geq \|\h_n\| = (n-1)(2n+1) \geq n^2-1.
    \]
    \item For $k=-j$ with $1\leq j\leq n-1$, taking the product of $\f_n^{-j}$ and $\h_n$ cancels precisely the initial $\r_1\cdots\r_j$ of $\h_n$, i.e.~there exists a reduced strand diagram~$\mathfrak{l}$ so that $\f_n^{-j}=\mathfrak{l}\cdot(\r_1\cdots\r_j)^{-1}$ and $\f_n^{-j}\h_n = \mathfrak{l}\cdot (\r_{j+1}\cdots \r_{n-1})$. Then
    \begin{align*}
    \bigl\|\f_n^{-j}\h_n\bigr\| &= \|\h_n\| + j\|\f_n\|  - 2\|\r_1\cdots\r_j\| \\
    &= \|\h_n\| + j(4n)  - 2\sum_{i=1}^{j} \bigl(4(n-i)+1\bigr) \\
    &= \|\h_n\|  - 2j(2n-2j-1).
    \end{align*}
This quantity is minimized when $j=\lfloor n/2\rfloor$, with a minimum value of $\|\h_n\|-n(n-1) = n^2-1$.\smallskip
\item For $k=-j$ with $j\geq n$, the product $\f_n^{-j}\h_n$ cancels all of $\h_n$, i.e.~$\f_n^{-j} = \bigl(\f_n^{-j}\h_n\bigr)\cdot \h_n^{-1}$.  It follows that
\begin{multline*}
\qquad\qquad \bigl\|\f_n^{-j}\h_n\bigr\| = j\|\f_n\|-\|\h_n\|  = j(4n) - (n-1)(2n+1) \\ \geq n(4n)-(n-1)(2n+1) = 2n^2+n+1 \geq n^2-1.\qquad
\end{multline*}
\end{itemize}
Thus $\bigl\|\f_n^k\h_n\bigr\|\geq n^2-1$ in all three cases, so the result follows.
\end{proof}

\begin{remark}\label{rem:ActualValues}
Using any of the known length formulas for $F$ \cite{BeBr,Guba,For} together with the analysis of centralizers in the proof of Theorem~\ref{thm:MainLowerBound}, it is possible to show that in fact $\ell(f_n)=\ell(g_n)=8n-5$ and $\cd(f_n,g_n)=\bigl\lceil 2n^2 - \frac{5}{2}n + 4\bigr\rceil$ for all $n\geq 3$, with $f_n^{-\lfloor n/2\rfloor}h_n$ being the unique minimum-length conjugator for $n\geq 4$. It follows that the conjugator length function for Thompson's group~$F$ satisfies
\[
\mathrm{CLF}(n) \geq 2\left\lfloor\frac{n+10}{16}\right\rfloor^2 - \frac{5}{2}\left\lfloor\frac{n+10}{16}\right\rfloor+4 \geq \frac{(n-15)^2+412}{128}
\]
for all $n\geq 38$. 
\end{remark}

This quadratic lower bound can also be made to work for $T$ and~$V$.  This depends on the following lemma.

\begin{lemma}\label{lem:SameCentralizer}
If $f\in F$ and the reduced annular strand diagram for $f$ is connected, then the centralizer of $f$ in $V$ is the same as the centralizer of $f$ in~$F$.
\end{lemma}
\begin{proof}
Let $k\in V$ so that $kf=fk$.  
Since the reduced annular strand diagram for $f$ is connected, we know that $f$ has no dyadic fixed points in the interval $(0,1)$ (see~\cite[Theorem~5.2]{BeMa}).  Let $0=p_0 <p_1 < \cdots < p_m=1$ be the fixed points of~$f$, which must be permuted by~$k$.  However, observe that for each $x\in[0,1]\setminus \{p_0,\ldots,p_m\}$, the full $f$-orbit $\{f^n(x)\}_{n\in\mathbb{Z}}$ has accumulation points at $p_{i-1}$ and $p_i$ for some~$i$.  Since $k$ maps full $f$-orbits to full $f$-orbits, it follows that $k(p_i)=p_i$ for all~$i$, and indeed $k$ maps each interval $[p_{i-1},p_i]$ to itself.

All that remains is to show that $k$ is order-preserving on each interval $(p_{i-1},p_i)$, and hence $k\in F$.  Let $\epsilon>0$ so that $k$ is linear on $(p_i-\epsilon,p_i]$, and let $p_{i-1}<x<y<p_i$.  Then there exists an $n\in\mathbb{Z}$ so that $p_i-\epsilon < f^n(x) < f^n(y) < p_i$.  Since $k$ is linear on $(p_i-\epsilon,p_i)$, it follows that $kf^n(x)<kf^n(y)$, so
\[
k(x) = f^{-n}kf^n(x) < f^{-n}kf^n(y) = k(y).\qedhere
\]
\end{proof}

\begin{theorem}
\label{lower}
In Thompson's group $T$ or $V$, there exists a constant $C>0$ so that
\[
\CLF(n) \geq Cn^2
\]
for all $n\in\mathbb{N}$.
\end{theorem}
\begin{proof}
Recall that elements of Thompson's group $V$ can also be represented by strand diagrams~(see~\cite{BeMa}).  If we fix a finite generating set for $V$, the length $\ell_{\scriptscriptstyle V}(f)$ and norm $\|f\|$ of an element $f\in V$ are related by the formula $\|f\| \leq m\,\ell_{\scriptscriptstyle V}(f)$, where $m$ is the maximum norm of any generator for $V$.  If $f\in F$, it follows from Proposition~\ref{prop:NormAndLength} that
\[
\frac{\ell_{\scriptscriptstyle F}(f)}{2m} \leq \frac{\|f\|}{m} \leq \ell_{\scriptscriptstyle V}(f) \leq \ell_{\scriptscriptstyle F}(f).
\]
That is, the embedding of $F$ into $V$ is quasi-isometric.  A similar argument shows that the embedding of $F$ into $T$ is quasi-isometric.

Now, by Lemma~\ref{lem:SameCentralizer} the centralizers of the elements $f_n$ are the same in $T$ or $V$ as they are in~$F$.  It follows that the conjugators from $f_n$ to $g_n$ in $T$ or $V$ are the same as they are in~$F$, and since the word lengths of the conjugators are the same up to a linear factor we obtain a quadratic lower bound on the conjugator length function.
\end{proof}

\bigskip
\newcommand{\arxiv}[1]{\href{https://arxiv.org/abs/#1}{arXiv:#1}}
\newcommand{\doi}[1]{\href{https://doi.org/#1}{\blue{doi:#1}}}
\bibliographystyle{plain}

\end{document}